\documentclass{amsproc}
\usepackage{amssymb}
\usepackage{amsfonts}

\theoremstyle{plain}
\newtheorem{acknowledgement}{Acknowledgement}

\newtheorem{lemma}{Lemma}

\newtheorem{remark}{Remark}

\newtheorem{theorem}{Theorem}
\numberwithin{equation}{section}

\begin{document}
\title{Lax-Halmos Type Theorems On $H^{p}$ Spaces}
\author{Niteesh Sahni}
\address{University of Delhi, Delhi, 110007\\
Shiv Nadar University, Dadri, Uttar Pradesh}
\email{niteeshsahni@gmail.com}
\author{Dinesh Singh}
\address{University of Delhi, Delhi, 110007}
\email{dineshsingh1@gmail.com}
\date{May 16, 2012}

\begin{abstract}
In this paper we characterize for $0<p\leq \infty $, the closed subspaces of 
$H^{p}$ that are invariant under multiplication by all powers of a finite
Blaschke factor $B$, except the first power. Our result clearly generalizes
the invariant subspace theorem obtained by Paulsen and Singh $[9]$ which has
proved to be the starting point of important work on constrained
Nevanlinna-Pick interpolation. Our method of proof can also be readily
adapted to the case where the subspace is invariant under all positive
powers of $B\left( z\right) .$ The two results are in the mould of the
classical Lax-Halmos Theorem and can be said to be Lax-Halmos type results
in the finitre multiplicity case for two commuting shifts and for a single
shift respectively.
\end{abstract}

\maketitle

\section{INTRODUCTION}

\noindent In recent times a great deal of interest has been generated in the
Banach algebra $H_{1}^{\infty }$ and subsequently also in a class of related
algebras in the context of problems dealing with invariant subspaces and
their use in solving Nevanlinna-Pick type interpolation problems for these
algebras. We refer to \cite{dh}, \cite{dprs}, \cite{jury}, \cite{mc}, \cite%
{paul}, and \cite{RS}. Note that $H_{1}^{\infty }=\{f(z)\in H^{\infty
}:f^{\prime }(0)=0\}$ is a closed subalgebra of $H^{\infty }$, the Banach
algebra of bounded analytic functions on the open unit disc. The starting
point, in the sequence of papers cited above, is an invariant subspace
theorem first proved by Paulsen and Singh in a special case \cite[Theorem 4.3%
]{paul} and subsequently in more general forms in \cite{dprs} and \cite%
{mrin2}. This invariant subspace result is crucial to the solutions of
interpolation problems of the Pick-Nevanlinna type as presented in the
papers cited above. This theorem characterizes the closed subspaces of the
Hardy spaces that are left invariant under the action of every element of
the algebra $H_{1}^{\infty }$. In this paper we present a far reaching
generalization of this invariant subspace theorem by presenting a complete
characterization of the invariant subspaces of the Banach algebra $%
H_{1}^{\infty }(B)=\{f(B(z)):f\in H_{1}^{\infty };B\text{ is a finite
Blaschke product}\}.$ In the special case where $B(z)=z$ we arrive at the
first or original invariant subspace theorem mentioned above for the Banach
algebra $H_{1}^{\infty }$. We also note that $H_{1}^{\infty }(B)$ stands for
the closed subalgebra of $H^{\infty }$ generated by $B^{2}$ and $B^{3}$. Our
paper also deals with a second and related problem of characterizing the
invariant subspaces on $H^{p}$ {\large of the algebra }$H^{\infty }(B)$
which consists of the Banach algebra generated by$B.$ For the case $p=2$,
this problem has been tackled in the far more general setting of de Branges
spaces in \cite{sinth} and in the same year, for the classical $H^{p}$spaces
for all values of $p\geq 1,$ this problem has been tackled in \cite{lance}.
However, for this second problem, we claim some novelty and completeness on
two counts; for one we have shown that in the case when $0<p<1$ we have an
explicit description of the invariant subspaces, and second, our proof is
more elementary and different from that in \cite{lance} since we do not use
their general inner-outer factorization theorem \cite[page 112]{lance}.

Finally, we wish to observe that the two main results presented in this
paper can be interpreted as being in the mould of the classical Lax-Halmos
Theorem in the case of finite multiplicity, see \cite{hal},\cite{hel}, and 
\cite{lax}, for two commuting shifts \ as represented by multiplication by $%
B^{2}$ and by $B^{3}$ and for a single shift represented by multiplication
by $B$ except that unlike the classical versions of the Lax-Halmos theorem
we work entirely in the scalar valued setting of the classical Hardy spaces
and our characterisations are also inside this scalar setting.

Let $\mathbb{D}$ denote the open unit disk, and let its boundary, the unit
circle, be denoted by $\mathbb{T}$. The Lebesgue space $L^{p}$ on the unit
circle is the collection of complex valued functions $f$ on the unit circle
such that $\int |f|^{p}dm$ is finite, where $dm$ is the normalized Lebesgue
measure on $\mathbb{T}$. The Hardy space $H^{p}$ is the closure in $L^{p}$
of the analytic polynomials. For $p\ge1$, $H^{p}$ can be viewed as the the
following closed subspace of $L^{p}$: 
\begin{equation*}
\left\{ f\in L^{p}:\int {fz^{n}}dm=0~\text{\ for all }n\geq 1\right\} .
\end{equation*}

\noindent {\large For }$1\leq p<\infty $, $H^{p}$ is a Banach space under
the norm 
\begin{equation*}
\Vert f\Vert _{p}=\left( \int {|f|^{p}}dm\right) ^{\frac{1}{p}}.
\end{equation*}

\noindent $H^{\infty }$ is a Banach algebra under the essential supremum
norm. The Hardy space $H^{2}$ turns out to be a Hilbert space under the
inner product 
\begin{equation*}
<f,g>=\int {f\overline{g}}dm.
\end{equation*}%
For a detailed account of $H^{p}$ spaces the reader can refer to \cite{dur}, 
\cite{gar}, \cite{hoff}, and \cite{koosis}.

\noindent By a finite Blaschke factor $B(z)$ we mean 
\begin{equation*}
B(z)=\prod_{j=1}^{n}\dfrac{z-\alpha _{j}}{1-\overline{\alpha _{j}}z},
\end{equation*}%
where $\alpha _{j}\in \mathbb{D}$. Throughout we shall assume $\alpha _{1}=0$
as this does not affect generality. The operator of multiplication by $B(z)$
denoted by $T_{B}$ is an isometric operator on $H^{p}$. We call a closed
subspace $M$ of $H^{p}$ to be $B$ -invariant if $T_{B}M\subset M$. By an
invariant subspace $M$ (in $H^{p}$) of any subalgebra $A$ of $H^{\infty }$
we mean a closed subspace of $H^{p}$ such that $fg\in M$ for all $f\ in$ $A$
and for all $g\in M$.

\noindent Let $B_{j}(z)$ denote the product of the first $j$ factors in $%
B(z) $: 
\begin{equation*}
B_{j}(z)=\prod_{i=1}^{j}\dfrac{z-\alpha _{i}}{1-\overline{\alpha _{i}}z}%
,~\alpha _{j}\in \mathbb{D}.
\end{equation*}

\section{Preliminary Results}

Note: \emph{We shall assume throughout this paper that the Blaschke product
B so indicated is fixed and has $n$ zeros that may not necessarily be
distinct}.

\begin{theorem}\label{thukral}
(\cite[Theorem 3.3]{sinth}) The collection $\ \{e_{jm}=(1-|\alpha
_{j+1}|^{2})^{\frac{1}{2}}(1-\overline{\alpha _{j+1}}z)^{-1}B_{j}B^{m}:0\leq
j\leq n-1,m=0,1,2,\ldots \}$ is an orthonormal basis for $H^{2}$.
Consequently $H^{2}=\sum_{j=0}^{n-1}\oplus e_{j0}H^{2}(B)$ where $H^{2}(B)$
stands for the closed linear span of $\{B^{m}:m=0,1,2,...\}$ in $H^2$.
\end{theorem}

\noindent Let $\left( \varphi _{1},\ldots ,\varphi _{r}\right) $ be an $r$
tuple of $H^{\infty }$ functions $\left( r\leq n\right) $. Suppose each $%
\varphi _{j}$ has the representation 
\begin{equation*}
\varphi _{j}=\sum\limits_{i=0}^{n-1}e_{i0}\varphi _{ij}.
\end{equation*}%
The matrix $A=\left( \varphi _{ij}\right) _{n\times r}$ is called the $B-$
matrix of $\left( \varphi _{1},\ldots ,\varphi _{r}\right) $. The matrix $A$
is called $B$-inner if $A^{\ast }A=I$. Suppose an $H^{\infty }$ function $%
\psi $ has the representation 
\begin{equation*}
\psi =\sum_{j=0}^{n-1}e_{j0}\psi _{j},
\end{equation*}%
for some $\psi _{0},\ldots ,\psi _{n-1}\in H^{2}(B)$, then $\psi $ is called 
$B$-inner if 
\begin{equation*}
\sum_{j=0}^{n-1}|\psi _{j}|^{2}=1~a.e.
\end{equation*}

\noindent It has been proved in \cite{sinth} that $\psi $ is $B$-inner if
and only if $\left\{ B^{m}\psi:m=0,1,...\right\} $ is an orthonormal set in $%
H^2$.

\begin{theorem}
(\cite[Theorem 4.1]{sinth}) Let $M$ be a $B$-invariant subspace of $H^{2}$.
Then there is an $r\leq n$ such that 
\begin{equation*}
M=\varphi _{1}H^{2}(B)\oplus \varphi _{2}H^{2}(B)\oplus \cdots \oplus
\varphi _{r}H^{2}(B)
\end{equation*}%
for some $B$-inner functions $\varphi _{1},\ldots ,\varphi _{r}$, and the $B$%
-matrix of $\left( \varphi _{1},\ldots ,\varphi _{r}\right) $ is $B$-inner.
Further, this representation is unique in the sense that if 
\begin{equation*}
M=\psi _{1}H^{2}(B)\oplus \psi _{2}H^{2}(B)\oplus \cdots \oplus \psi
_{s}H^{2}(B)
\end{equation*}%
then $r=s$, and $\varphi _{i}=\sum\limits_{j=1}^{r}\alpha _{ij}\psi _{j}$
for scalars $\alpha _{ij}$ such that the matrix $(\alpha _{ij})$ is unitary.
\end{theorem}

We shall also make use of the following results to establish certain key
facts that are central to the proof of the main results.

\begin{lemma}
(\cite[Proposition 3]{lance}) Let $1\leq p\leq \infty $ and let $\varphi
_{1},\ldots ,\varphi _{k}$ $\left( k\leq n\right) $ be $B$-inner functions.
Then for any $f\in H^{\infty }$ such that $f\left( z\right)
=\sum\limits_{i=1}^{k}\varphi _{i}\left( z\right) f_{i}\left( B\left(
z\right) \right) $, there exist constants $C_{i,p}$, $i=1,\ldots ,k$ such
that $\left\Vert f_{i}\right\Vert _{p}\leq C_{i,p}\left\Vert f\right\Vert
_{p}$.
\end{lemma}

\begin{lemma}
For $1\leq p<2$, we can write 
\begin{equation*}
H^{p}=e_{00}H^{p}\left( B\right) \oplus e_{10}H^{p}\left( B\right) \oplus
\cdots \oplus e_{n-1,0}H^{p}\left( B\right) ,
\end{equation*}%
where each $e_{j0}$ is as in Theorem $1$ for each $j$.
\end{lemma}

\begin{proof}
It is trivial to note that 
\begin{equation*}
e_{00}H^{p}\left( B\right) \oplus e_{10}H^{p}\left( B\right) \oplus \cdots
\oplus e_{n-1,0}H^{p}\left( B\right) \subset H^{p}.
\end{equation*}

\noindent To establish the opposite inclusion take an arbitrary element $%
f\in H^{p}$. Since $H^{\infty }$ is dense in $H^{p}$, there exists a
sequence $\{f_{k}\}$ of $H^{\infty }$ functions that converges to $f$ in the
norm of $H^{p}$. In view of Theorem 1 we can write 
\begin{equation}
f_{k}=e_{00}f_{k}^{(1)}+e_{10}f_{k}^{(2)}+\cdots +e_{n-1,0}f_{k}^{(n)},
\end{equation}%
where $f_{k}^{(j)}\in H^{2}\left( B\right) $ for all $j=1,...,n$. By Lemma
1, we have for all $j=1,\ldots n$, the estimate 
\begin{equation}
\left\Vert f_{k}^{(j)}\right\Vert _{p}\leq D_{kj}\left\Vert f_{k}\right\Vert
_{p}
\end{equation}%
for some constants $D_{kj}$. Equation $(2.2)$ implies that $\left\{
f_{k}^{(j)}\right\} $ is a Cauchy sequence for all $j=1,\ldots ,n$, and
hence $f_{k}^{(j)}\rightarrow f^{(j)}$ in $H^{p}\left( B\right) $. Therefore 
$f_{k}\rightarrow $ $e_{00}f^{(1)}+e_{10}f^{(2)}+\cdots +e_{n-1,0}f^{(n)}$
as $k\rightarrow \infty $ in $H^{p}$. Hence $f=e_{00}f^{(1)}+e_{10}f^{(2)}+%
\cdots +e_{n-1,0}f^{(n)}$. This completes the proof of the assertion.
\end{proof}

\begin{lemma}
(\cite[Lemma 4.1]{gam}) Suppose $\left\{ f_{n}\right\} _{n=1}^{\infty }$ is
a sequence of $H^{p}$ functions, $p>2$, which converges to an $H^{p}$
function $f$ in the $H^{2}$ norm. Then there exists a sequence $\left\{
g_{n}\right\} _{n=1}^{\infty }$ of $H^{\infty }$ functions such that $%
g_{n}f_{n}\rightarrow f$ in the $H^{p}$ norm (weak-star convergence when $%
p=\infty $). Further the sequence $\left\{ g_{n}\right\} _{n=1}^{\infty }$
is uniformly bounded, and converges to the constant function $1~a.e.$
\end{lemma}

\begin{lemma}
Let $p>2$. Suppose an $H^{p}$ function $f$ is of the form $f=\varphi
_{1}h_{1}+\cdots +\varphi _{n}h_{n}$, where $h_{1},\ldots ,h_{n}\in
H^{2}\left( B\right) $, and $\varphi _{1},\ldots ,\varphi _{n}$ are $B$%
-inner , then $h_{1},\ldots ,h_{n}$ belong to $H^{p}\left( B\right)$.
\end{lemma}

\begin{proof}
Let 
\begin{equation*}
f=\varphi _{1}h_{1}+\varphi _{2}h_{2}+\cdots +\varphi _{r}h_{r},
\end{equation*}%
where $h_{1},h_{2},\ldots ,h_{r}\in H^{2}\left( B\right) $, and $\varphi
_{1},\ldots ,\varphi _{n}$ are $B-$ inner. Also $f$ can be identified with a
bounded linear functional $F_{f}\in L_{q}^{\ast }${\large \ }$\left( 1\leq
q<2\right) $ such that 
\begin{equation*}
F_{f}\left( g\right) =\int fg~\text{\ and }\left\vert F_{f}\left( g\right)
\right\vert \leq \delta \left\Vert g\right\Vert _{q}\text{ for all }g\in
L_{q}\text{ for some }\delta >0\text{.}
\end{equation*}%
Now for any $l\in span\left\{ 1,B,B^{2},...\right\} $ such that $l$ is a
polynomial in $B$, we have 
\begin{eqnarray*}
\left\vert \int f~\overline{\varphi _{1}l}\right\vert  &\leq &\delta
\left\Vert \varphi _{1}l\right\Vert _{q} \\
&\leq &\delta _{1}\left\Vert l\right\Vert _{q}
\end{eqnarray*}%
for some $\delta _{1}>0$. Also note that 
\begin{eqnarray*}
&&\left\vert \int f~\overline{\varphi _{1}l}\right\vert  \\
&=&\left\vert \int \left( \varphi _{1}h_{1}+\varphi _{2}h_{2}+\cdots
+\varphi _{r}h_{r}\right) ~\overline{\varphi _{1}l}\right\vert  \\
&=&\left\vert \int \varphi _{1}h_{1}\overline{\varphi _{1}l}+\int \varphi
_{2}h_{2}\overline{\varphi _{1}l}+\cdots +\int \varphi _{r}h_{r}~\overline{%
\varphi _{1}l}\right\vert  \\
&=&\left\vert \int \varphi _{1}h_{1}\overline{\varphi _{1}l}\right\vert  \\
&=&\left\vert \int h_{1}\overline{l}\right\vert .
\end{eqnarray*}%
Here $\int \varphi _{j}h_{j}\overline{\varphi _{1}l}=0$, $j=2,...,r$ because 
$\varphi _{1}H^{2}\left( B\right) \perp \varphi _{j}H^{2}\left( B\right) $,
and $\int \varphi _{1}h_{1}\overline{\varphi _{1}l}=\int h_{1}\overline{l}$
because $\varphi _{1}$ is $B-$ inner. Therefore, 
\begin{equation*}
\left\vert \int h_{1}\overline{l}\right\vert \leq \delta _{1}\left\Vert
l\right\Vert _{q}.
\end{equation*}%
Now any analytic polynomial $k\in L_{q}$ can be written as 
\begin{equation*}
k=e_{00}k_{1}\left( B\right) +\cdots +e_{r-10}k_{r}\left( B\right) 
\end{equation*}%
and 
\begin{eqnarray}
\left\vert \int h_{1}k\right\vert  &=&\left\vert \int h_{1}k_{1}+\int
e_{10}h_{1}k_{2}+\cdots +\int e_{r-10}h_{1}k_{r}\right\vert   \notag \\
&\leq &\left\vert \int h_{1}k_{1}\right\vert +\left\vert \int
e_{10}h_{1}k_{2}\right\vert +\cdots +\left\vert \int
e_{r-10}h_{1}k_{r}\right\vert   \notag \\
&=&\left\vert \int h_{1}k_{1}\right\vert .
\end{eqnarray}%
It is easily checked that all integrals in the above equation except the
first integral shall be zero. We show this by examining one of the above
integrals in question:\newline
Let $h_{1}=\alpha _{0}+\alpha _{1}B+\alpha _{2}B^{2}+\cdots ,~k_{2}=\beta
_{0}+\beta _{1}B+\beta _{2}B^{2}+\cdots $. \newline
Now, 
\begin{eqnarray*}
\int e_{10}h_{1}k_{2} &=&\left\langle e_{10}h_{1},\overline{k_{2}}%
\right\rangle \newline
\\
&=&\left\langle \dfrac{z}{1-\overline{\alpha _{2}}z}\left( \alpha
_{0}+\alpha _{1}B+\alpha _{2}B^{2}+\cdots \right) ,\overline{\beta _{0}}+%
\overline{\beta _{1}}\overline{B}+\overline{\beta _{2}}\overline{B}%
^{2}+\cdots \right\rangle  \\
&=&0.
\end{eqnarray*}%
Let $k_{1}=\gamma _{0}+\gamma _{1}B+\gamma _{2}B^{2}+\cdots $so that 
\begin{eqnarray*}
\left\vert \int h_{1}k_{1}\right\vert  &=&\left\vert \left\langle h_{1},%
\overline{k_{1}}\right\rangle \right\vert  \\
&=&\left\vert \alpha _{0}\right\vert \left\vert \gamma _{0}\right\vert  \\
&=&\left\vert \int h_{1}\right\vert \left\vert \int k_{1}\right\vert  \\
&=&\left\vert \int h_{1}\right\vert \left\vert \int k\right\vert  \\
&\leq &A\int \left\vert k\right\vert \leq A\left\Vert k\right\Vert _{q}
\end{eqnarray*}%
where $A=\int \left\vert h_{1}\right\vert $. Thus $h_{1}$acts as a bounded
linear functional on the space of polynomials, and so it can be extended to
a bounded linear functional on $L_{q}$. Call this extension as $F$. So $%
F\left( g\right) =\int h_{1}g$ for all $g\in L_{q}$. But $F\in L_{q}^{\ast }$
implies that there exists $G\in L_{p}$ such that $F\left( g\right) =\int Gg$
for all $g\in L_{q}$. In particular we have 
\begin{equation*}
\int \left( G-h_{1}\right) z^{n}=0\text{ for all }n\in \mathbb{Z}\text{.}
\end{equation*}%
Hence $h_{1}=G\in L^{p}$. In a similar fashion we get $h_{2},...,h_{r}\in
L^{p}$.
\end{proof}

\noindent

\section{The $B^{2}$ and $B^{3}$ invariant subspaces of $H^{p}$}

\begin{theorem}
Let $M$ be a closed subspace of $H^{p}$, $0<p\leq \infty ,$ such that $M$ is
invariant under $H_{1}^{\infty }\left( B\right) $ but not invariant under $%
H^{\infty }\left( B\right) $. Then there exist $B-$ inner functions $%
J_{1},\ldots ,J_{r}$ $(r\leq n)$ such that 
\begin{equation*}
M=\left( \sum\limits_{j=1}^{k}\oplus \left\langle \varphi _{j}\right\rangle
\right) \oplus \sum_{l=1}^{r}\oplus B^{2}J_{l}H^{p}\left( B\right) 
\end{equation*}%
where $k\leq 2r-1$, and for all $j=1,2,...,k,$ $\varphi _{j}=(\alpha
_{1j}+\alpha _{2j}B)J_{1}+(\alpha _{3j}+\alpha _{4j}B)J_{2}+...+(\alpha
_{2r-1,j}+\alpha _{2r,j}B)J_{r}.$ 

\begin{remark}
The proof shall also show that the matrix $A=\left( \alpha _{ij}\right)
_{2r\times k}$ satisfies $A^{\ast }A=I$, and $\alpha _{st}\neq 0$ for some $%
(s,t)\in \{1,3,\ldots ,2r-1\}\times \{1,2,\ldots ,k\}$. Also, when $0<p<1,$
the right hand side should be interpreted as being the closure of the sum in
the $H^{p}$ metric.
\end{remark}
\end{theorem}

\begin{proof}
\textbf{The case }$p=2$. Using the initial line of argument as in \cite{dprs}
we define $M_{1}=\overline{H^{\infty }\left( B\right) \cdot M}$. It is
easily seen that $M_{1}$ is a $B-$invariant subspace of $H^{2}$. Observe
that 
\begin{equation}
M_{1}\supset M\supset H_{1}^{\infty }\left( B\right) \cdot M\supset 
\overline{B^{2}H^{\infty }\left( B\right) \cdot M}=B^{2}M_{1}.
\end{equation}%
Therefore 
\begin{equation}
B^{2}M_{1}\subset M\subset M_{1}
\end{equation}%
Note that all containments in the above equation are strict. For if $M=M_{1}$
or $M=B^{2}M_{1}$ it would then mean that $M$ is invariant under $H^{\infty
}\left( B\right) $, which is a contradiction. By Theorem 2, there exist $B-$
inner functions $J_{1},\ldots ,J_{r}$ $\left( r\leq n\right) $ such that 
\begin{eqnarray*}
M_{1} &=&J_{1}H^{2}\left( B\right) \oplus \cdots \oplus J_{r}H^{2}\left(
B\right)  \\
&=&\left( \left\langle J_{1}\right\rangle \oplus \left\langle
BJ_{1}\right\rangle \oplus B^{2}J_{1}H^{2}(B)\right) \oplus \cdots \oplus
\left( \left\langle J_{r}\right\rangle \oplus \left\langle
BJ_{r}\right\rangle \oplus B^{2}J_{r}H^{2}(B)\right)  \\
&=&\left( \left\langle J_{1}\right\rangle \oplus \left\langle
BJ_{1}\right\rangle \oplus \cdots \oplus \left\langle J_{r}\right\rangle
\oplus \left\langle BJ_{r}\right\rangle \right) \oplus B^{2}M_{1}.
\end{eqnarray*}%
So $M_{1}\ominus B^{2}M_{1}$ has dimension $2r$, and hence $M\ominus
B^{2}M_{1}$ has dimension $k$, where $k\leq 2r-1$. Let $\varphi _{1},\ldots
,\varphi _{k}$ be an orthonormal basis for $M\ominus B^{2}M_{1}$. Now 
\begin{eqnarray*}
M &=&\left[ M\ominus B^{2}M_{1}\right] \oplus B^{2}M_{1} \\
&=&\left( \sum\limits_{j=1}^{k}\oplus \left\langle \varphi _{j}\right\rangle
\right) \oplus B^{2}J_{1}H^{2}\left( B\right) \oplus \cdots \oplus
B^{2}J_{r}H^{2}\left( B\right) .
\end{eqnarray*}%
Since $\varphi _{j}\in M\ominus B^{2}M_{1}\subset M_{1}\ominus B^{2}M_{1}$,
we see that each $\varphi _{j}$ is of the form $\alpha _{1j}J_{1}+\alpha
_{2j}J_{1}B+\alpha _{3j}J_{2}+\alpha _{4j}J_{2}B+\cdots +\alpha
_{2r-1,j}J_{r}+\alpha _{2r,j}J_{r}B$. The conditions $\left\Vert \varphi
_{j}\right\Vert _{2}^{2}=1$ and $\left\langle \varphi _{j},\varphi
_{i}\right\rangle =0$ for $j\neq i$, imply that $\left\vert \alpha
_{1j}\right\vert ^{2}+\left\vert \alpha _{2j}\right\vert ^{2}+\cdots
+\left\vert \alpha _{2r,j}\right\vert ^{2}=1$, and $\alpha _{1j}\overline{%
\alpha _{1i}}+\alpha _{2j}\overline{\alpha _{2i}}+\cdots +\alpha _{2r,j}%
\overline{\alpha _{2r,i}}=0$. In addition the $k$ tuples $\left( \alpha
_{i1},\alpha _{i2},\ldots ,\alpha _{ik}\right) $, $i=1,3,\ldots ,2r-1$
cannot be simultaneously zero, otherwise $M$ would become $B-$ invariant.
%Thus the matrix $A=\left( \alpha _{ij}\right) _{2r\times k}$ has orthonormal
%columns, and the submatrix $\left( \alpha _{2i-1,j}\right) _{r\times k}$
%must be non-zero.

\textbf{The case }$0<p<1$. Observe that every $H^{p}$ function $f$, can be
written as $f=IO$, where $I$ is an inner function and $O\in H^{p}$ is an
outer function. Choose $n$ such that $2^{n}p>2$, so that we can express $f$
as a product of $H^{2}$ functions: 
\begin{equation*}
f=IO^{\frac{1}{2^{n}}}O^{\frac{1}{2^{n}}}\cdots O^{\frac{1}{2^{n}}}.
\end{equation*}%
We first show that $M\cap H^{2}\neq \lbrack 0]$. Let $0\neq f\in M$. Then $f$
can be written as 
\begin{equation*}
f=f_{1}f_{2}\cdots f_{m},
\end{equation*}%
where $f_{1}$, $f_{2}$,\ldots ,$f_{m}\in H^{2}$. In view of Theorem \ref%
{thukral}, we can express each $f_{l}$ as 
\begin{equation*}
f_{l}=e_{00}g_{1}^{(l)}+\cdots +e_{r0}g_{r}^{(l)},
\end{equation*}%
for some $g_{1}^{(l)},\ldots ,g_{r}^{(l)}\in H^{2}\left( B^{2}\right) $. It
is known that the operator $T:H^{2}\longrightarrow H^{2}$ defined by $%
Th=h\left( B^{2}\left( z\right) \right) $ is an isometry, and its range is $%
H^{2}\left( B^{2}\right) $ (see \cite{barb}). So for each $g_{j}^{(l)}$,
there exists $k_{j}^{(l)}\in H^{2}$such that $g_{j}^{(l)}=k_{j}^{(l)}\left(
B^{2}\left( z\right) \right) $. Define 
\begin{equation*}
q_{jl}\left( z\right) :=\exp \left\{ \frac{-\left\vert k_{j}^{\left(
l\right) }\left( z\right) \right\vert ^{\frac{1}{2}}-i\widetilde{\left\vert
k_{j}^{\left( l\right) }\left( z\right) \right\vert ^{\frac{1}{2}}}}{2}%
\right\} .
\end{equation*}%
Here $\sim $ denotes the harmonic conjugate. Then $\left\vert q_{jl}\left(
z\right) \right\vert \leq 1$, and thus $h_{l}\left( z\right) :=q_{1l}\left(
z\right) q_{2l}\left( z\right) \cdots q_{rl}\left( z\right) \in H^{\infty }$.

Note that 
\begin{eqnarray*}
h_{l}\left( B^{2}(z)\right) f_{l}\left( z\right)  &=&e_{00}h_{l}\left(
B^{2}(z)\right) g_{1}^{(l)}\left( z\right) +\cdots +e_{r0}h_{l}\left(
B^{2}(z)\right) g_{r}^{(l)}\left( z\right)  \\
&=&e_{00}h_{l}\left( B^{2}(z)\right) k_{1}^{(l)}\left( B^{2}\left( z\right)
\right) +\cdots +e_{r0}h_{l}\left( B^{2}(z)\right) k_{r}^{(l)}\left(
B^{2}\left( z\right) \right) ,
\end{eqnarray*}%
which clearly belongs to $H^{\infty }$. This implies that $h_{1}\left(
B^{2}\left( z\right) \right) \cdots h_{m}\left( B^{2}\left( z\right) \right)
f=h_{1}\left( B^{2}\left( z\right) \right) f_{1}\cdots h_{m}\left(
B^{2}\left( z\right) \right) f_{m}\in H^{\infty }$. Since $h_{1}\left(
B^{2}\left( z\right) \right) \cdots h_{m}\left( B^{2}\left( z\right) \right)
\in H^{\infty }$ its Cesaro means $\left\{ p_{n}\left( B^{2}\right) \right\} 
$, which is a sequence of polynomials, shall converge to $h_{1}\left(
B^{2}\left( z\right) \right) \cdots h_{m}\left( B^{2}\left( z\right) \right)
a.e.$. Hence, by the Dominated Convergence Theorem, we see that $p_{n}\left(
B^{2}\right) f\rightarrow h_{1}\left( B^{2}\left( z\right) \right) \cdots
h_{m}\left( B^{2}\left( z\right) \right) f$ in $H^{p}$. Therefore, $%
h_{1}\left( B^{2}\left( z\right) \right) \cdots h_{m}\left( B^{2}\left(
z\right) \right) f\in M$, because $M$ is invariant under $B^{2}$. This
establishes $M\cap H^{2}\neq \{0\}.$ Next we claim that $M\cap H^{2}$ is
dense in $M$. The density will also imply that $M\cap H^{2}$ is not $B-$
invariant, otherwise it would force $M$ to be $B-$ invariant, which is not
possible. It is trivial to note that $\overline{M\cap H^{2}}\subseteq M$
(bar denotes closure in $H^{p}$). Let $f\in M$. We can express $f$ as 
\begin{equation*}
f=f_{1}f_{2}\cdots f_{2^{m}},
\end{equation*}%
where each $f_{l}\in H^{2}$, and $m$ is chosen so that $2^{m}p>2$. As argued
previously we can express each $f_{l}$ as 
\begin{equation*}
f_{l}=e_{00}k_{1}^{(l)}\left( B^{2}\left( z\right) \right) +\cdots
+e_{r0}k_{r}^{(l)}\left( B^{2}\left( z\right) \right) ,
\end{equation*}%
for certain $k_{1}^{(l)},\ldots ,k_{r}^{(l)}\in H^{2}$. Define 
\begin{equation*}
q_{n}^{(jl)}(z)=exp\left( \dfrac{-|k_{j}^{\left( l\right) }\left( z\right)
|^{\frac{1}{2}}-i\widetilde{|k_{j}^{\left( l\right) }\left( z\right) |}^{%
\frac{1}{2}}}{n}\right) 
\end{equation*}%
($\sim $ denotes the harmonic conjugate which exists for $L^{2}$ functions).
Then $q_{n}^{(jl)}\in H^{\infty }$ and $\left\vert q_{n}^{(jl)}\right\vert
~\leq 1$. For each $l=1,...,2^{m}$, the function $%
h_{n}^{(l)}(z)=q_{n}^{(1l)}(z)\cdots q_{n}^{(rl)}\left( z\right) $ belongs
to $H^{\infty }$ and $h_{n}^{(l)}\left( B^{2}\left( z\right) \right) $
multiplies $f_{l}${\large \ }into{\large \ }$H^{\infty }$. This implies that 
\begin{equation*}
h_{n}^{(1)}(B^{2}\left( z\right) )\cdots h_{n}^{(2^{m})}(B^{2}\left(
z\right) )f\in H^{\infty }
\end{equation*}%
Since $h_{n}^{(1)}(B^{2}\left( z\right) )\cdots h_{n}^{(2^{m})}(B^{2}\left(
z\right) )\rightarrow 1~a.e.$, we have 
\begin{equation*}
h_{n}^{(1)}(B^{2}\left( z\right) )\cdots h_{n}^{(2^{m})}(B^{2}\left(
z\right) )f\rightarrow f~a.e.
\end{equation*}%
so that 
\begin{equation*}
\left\vert h_{n}^{(1)}(B^{2}\left( z\right) )\cdots
h_{n}^{(2^{m})}(B^{2}\left( z\right) )f-f\right\vert ^{p}\rightarrow 0\text{ 
}a.e
\end{equation*}%
Moreover%
\begin{equation*}
\left\vert h_{n}^{(1)}(B^{2}\left( z\right) )\cdots
h_{n}^{(2^{m})}(B^{2}\left( z\right) )f-f\right\vert ^{p}\leq
2^{p}\left\vert f\right\vert ^{p}
\end{equation*}%
so by the Dominated Convergence Theorem, 
\begin{equation*}
h_{n}^{(1)}(B^{2}\left( z\right) )\cdots h_{n}^{(2^{m)}}(B^{2}\left(
z\right) )f\rightarrow f\in H^{p}.
\end{equation*}%
We claim that $h_{n}^{(1)}(B^{2})\cdots h_{n}^{(2^{m})}(B^{2})f\in M$. To
prove this claim we proceed as follows. For each $n$, there exists a
sequence of polynomials $\{p_{k}^{(n)}\}$ such that%
\begin{equation*}
p_{k}^{(n)}(B^{2}(z))\rightarrow h_{n}^{(1)}(B^{2}\left( z\right) )\cdots
h_{n}^{(2^{m})}(B^{2}\left( z\right) )
\end{equation*}%
The sequence $\{p_{k}^{(n)}\}$ is the Cesaro means of $%
h_{n}^{(1)}...h_{n}^{(2^{m})}$ and it converges boundedly and pointwise. It
is then easy to see by means of the Dominated Convergence Theorem that $%
p_{k}^{(n)}(B^{2})f$ converges to $h_{n}^{(1)}(B^{2})\cdots
h_{n}^{(2^{m})}(B^{2})f$ in $H^{p}$. The claim now follows in view of the
fact that $M$ is invariant under $B^{2}$ and the fact that $%
p_{k}^{(n)}(B^{2})f\in M$\ . By the validity of the result for the case $p=2$%
, we have 
\begin{equation*}
M\cap H^{2}=\left( \sum\limits_{j=1}^{2r-1}\oplus \left\langle \varphi
_{j}\right\rangle \right) \oplus B^{2}J_{1}H^{2}\left( B\right) \oplus
\cdots \oplus B^{2}J_{r}H^{2}\left( B\right) ,
\end{equation*}%
and hence 
\begin{eqnarray*}
M &=&\overline{\left( \sum\limits_{j=1}^{2r-1}\oplus \left\langle \varphi
_{j}\right\rangle \right) \oplus B^{2}J_{1}H^{2}\left( B\right) \oplus
\cdots \oplus B^{2}J_{r}H^{2}\left( B\right) } \\
&=&\left( \sum\limits_{j=1}^{2r-1}\oplus \left\langle \varphi
_{j}\right\rangle \right) \oplus B^{2}\left( \overline{J_{1}H^{2}\left(
B\right) \oplus \cdots \oplus J_{r}H^{2}\left( B\right) }\right) 
\end{eqnarray*}%
(bar denotes closure in $H^{p}$). We can easily see that $%
B^{2}(J_{1}H^{p}(B)\oplus ...\oplus J_{r}H^{p}(B))\subset B^{2}\left( 
\overline{J_{1}H^{2}\left( B\right) \oplus \cdots \oplus J_{r}H^{2}\left(
B\right) }\right) \subset B^{2}\left( \overline{J_{1}H^{p}\left( B\right)
\oplus \cdots \oplus J_{r}H^{p}\left( B\right) }\right) $ and so upon taking
the closure in $H^{p}$ of all three subspaces we shall get equality
throughout so that\ $B^{2}\left( \overline{J_{1}H^{2}\left( B\right) \oplus
\cdots \oplus J_{r}H^{2}\left( B\right) }\right) =B^{2}\left( \overline{%
J_{1}H^{p}\left( B\right) \oplus \cdots \oplus J_{r}H^{p}\left( B\right) }%
\right) $\ and this gives us the characterisation for the case $0<p<1.$

\textbf{The case }$1\leq p<2$. The arguments and conclusions above in the
case $0<p<1$ are also valid for this case and so certainly%
\begin{equation*}
M=\left(\sum\limits_{j=1}^{2r-1}\oplus \left\langle \varphi
_{j}\right\rangle \right)\oplus \left( \overline{J_{1}H^{p}\left( B\right)
\oplus \cdots \oplus J_{r}H^{p}\left( B\right) }\right)
\end{equation*}
where the bar denotes closure in $H^{p},1\leq p<2$. Let $N=\overline{%
J_{1}H^{p}\left( B\right) \oplus \cdots \oplus J_{r}H^{p}\left( B\right) }$
so that $M=(\sum\limits_{j=1}^{2r-1}\oplus \left\langle \varphi
_{j}\right\rangle )\oplus N$. Then as a closed subspace of $H^{p}$, $N$ is
invariant under multiplication by $B$. It can be verified that 
\begin{equation*}
N\cap H^{2}=J_{1}H^{2}\left( B\right) \oplus \cdots \oplus J_{r}H^{2}\left(
B\right) .
\end{equation*}%
Any arbitrary $g\in J_{i}H^{p}\left( B\right) $ can be written as $g=J_{i}f$%
, for some $f\in H^{p}\left( B\right) $. Then the Cesaro means of $f$
denoted by the sequence of polynomials, $\left\{ p_{n}\right\} $, is such
that $p_{n}(z)\rightarrow f(z)$ in $H^{p}$. Hence $p_{n}(B)\rightarrow f(B)$
in $H^{p}.$ But $J_{i}\in H^{\infty }$, so $p_{n}(B)J_{i}\rightarrow
J_{i}f(B)$ in $H^{p}$. Because $N$ is $B-$ invariant, we have $%
p_{n}(B)J_{i}\in N$, and the fact that $N$ is closed implies that $J_{i}f\in
N$. This establishes that $J_{1}H^{p}\left( B\right) \oplus \cdots \oplus
J_{r}H^{p}\left( B\right) \subset N$. Now we establish the inclusion in the
other direction. In a fashion, similar to as shown above, for any $f\in N$,
we can construct an outer function $K\in H^{\infty }$ such that $Kf\in N\cap
H^{2}$. Therefore, 
\begin{equation}
Kf=J_{1}h_{1}+J_{2}h_{2}+\cdots +J_{r}h_{r},
\end{equation}%
for some uniquely determined $h_{1},h_{2},\ldots ,h_{r}\in H^{2}(B)\subset
H^{p}\left( B\right) $. Since $f\in H^{p}$, by Lemma $2$, we can express it
uniquely as 
\begin{equation}
f=e_{00}f_{1}+e_{10}f_{2}+\cdots +e_{n-1,0}f_{n},
\end{equation}%
for some $f_{1},\ldots ,f_{n}\in H^{p}\left( B\right) $. Therefore, 
\begin{equation}
Kf=e_{00}Kf_{1}+e_{10}Kf_{2}+\cdots +e_{n-1,0}Kf_{n}.
\end{equation}%
Because $J_{1},J_{2},\ldots ,J_{r}$ are $B-$ inner, we can write 
\begin{equation*}
\begin{array}{c}
J_{1}=e_{00}\varphi _{10}+e_{10}\varphi _{11}+\cdots +e_{n-1,0}\varphi
_{1,n-1} \\ 
J_{2}=e_{00}\varphi _{20}+e_{10}\varphi _{21}+\cdots +e_{n-1,0}\varphi
_{2,n-1} \\ 
\vdots \\ 
J_{r}=e_{00}\varphi _{r0}+e_{10}\varphi _{r1}+\cdots +e_{n-1,0}\varphi
_{r,n-1},%
\end{array}%
\end{equation*}%
where the $B-$ matrix $\left( \varphi _{ij}\right) _{r\times n}$ satisfies $%
\left( \varphi _{ij}\right) _{r\times n}\left( \overline{\varphi _{ji}}%
\right) _{n\times r}=I$.\newline
Equation $3.3$ now becomes 
\begin{eqnarray}
Kf &=&e_{00}\left( \varphi _{10}h_{1}+\varphi _{20}h_{2}+\cdots +\varphi
_{r0}h_{r}\right) +  \notag \\
&&e_{10}\left( \varphi _{11}h_{1}+\varphi _{21}h_{2}+\cdots +\varphi
_{r1}h_{r}\right)  \notag \\
&&+\cdots +  \notag \\
&&e_{n-1,0}\left( \varphi _{1,n-1}h_{1}+\varphi _{2,n-1}h_{2}+\cdots
+\varphi _{r,n-1}h_{r}\right)
\end{eqnarray}%
From equations $(3.5)$ and $(3.6)$ we see that 
\begin{equation*}
\begin{array}{c}
Kf_{1}=\varphi _{10}h_{1}+\varphi _{20}h_{2}+\cdots +\varphi _{r0}h_{r} \\ 
\vdots \\ 
Kf_{n}=\varphi _{1,n-1}h_{1}+\varphi _{2,n-1}h_{2}+\cdots +\varphi
_{r,n-1}h_{r}.%
\end{array}%
\end{equation*}%
This in matrix form can be written as 
\begin{equation}
\left( Kf_{i}\right) _{1\times n}=\left( h_{i}\right) _{1\times r}\left(
\varphi _{ij}\right) _{r\times n}
\end{equation}%
Taking the conjugate transpose we get 
\begin{equation}
\left( \overline{Kf_{i}}\right) _{n\times 1}=\left( \overline{\varphi _{ji}}%
\right) _{n\times r}\left( \overline{h_{i}}\right) _{r\times 1}
\end{equation}%
By multiplying equations $(3.7)$ and $(3.8)$ we get: 
\begin{eqnarray*}
\left\vert \frac{h_{1}}{K}\right\vert ^{2}+\cdots +\left\vert \frac{h_{r}}{K}%
\right\vert ^{2} &=&\left\vert f_{1}\right\vert ^{2}+\cdots +\left\vert
f_{n}\right\vert ^{2} \\
&\leq &\left( \left\vert f_{1}\right\vert +\cdots +\left\vert
f_{n}\right\vert \right) ^{2}.
\end{eqnarray*}%
Thus for $j=1,...,r$, we have 
\begin{equation*}
\left\vert \frac{h_{j}}{K}\right\vert \leq \left\vert f_{1}\right\vert
+\cdots +\left\vert f_{n}\right\vert
\end{equation*}%
and this clearly implies that $\dfrac{h_{j}}{K}\in L^{p}$. Because $K$ is
outer, we have $\dfrac{h_{j}}{K}\in H^{p}$. Then from (3.3) we conclude that 
$f$ is in $J_{1}H^{p}\left( B\right) \oplus \cdots \oplus J_{r}H^{p}\left(
B\right) $ so that $N\subset J_{1}H^{p}\left( B\right) \oplus \cdots \oplus
J_{r}H^{p}\left( B\right) $ and so $N=J_{1}H^{p}\left( B\right) \oplus
\cdots \oplus J_{r}H^{p}\left( B\right) $ and this then implies that $%
M=\left(\sum\limits_{j=1}^{k}\oplus \left\langle \varphi _{j}\right\rangle
\right)\oplus \sum_{l=1}^{r}\oplus B^{2}J_{l}H^{p}\left( B\right)$.

\textbf{The case }$1<p\leq \infty $. Let $M_{1}=\overline{M}$ denote the
closure of $M$ in $H^{2}$. Suppose $M_{1}$ is invariant under multiplication
by $B(z)$, then by Theorem 2, we can write 
\begin{equation*}
M_{1}=\varphi _{1}H^{2}\left( B\right) \oplus \cdots \oplus \varphi
_{n}H^{2}\left( B\right) ,
\end{equation*}%
for some $B-$ inner functions $\varphi _{1},\ldots ,\varphi _{n}$. It
follows that any element $f\in M$ can be written as 
\begin{equation*}
f=\varphi _{1}h_{1}+\cdots +\varphi _{n}h_{n},
\end{equation*}%
for some $h_{1},\ldots ,h_{n}\in H^{2}\left( B\right) $. By Lemma 2, $%
h_{j}\in H^{p}${\large \ }(in fact $h_{j}\in ${\large \ }$H^{p}\left(
B\right) $). We claim that $\Phi _{k}=\varphi _{k}h_{k}\in M$. Since $\Phi
_{k}\in \overline{M}$, so there exists a sequence $\left\{
h_{l}^{(k)}\right\} _{l=1}^{\infty }\subset M$ such that $%
h_{l}^{(k)}\longrightarrow \Phi _{k}$ in $H^{2}$ as $l\rightarrow \infty $.
Moreover we can write 
\begin{equation*}
h_{l}^{(k)}=e_{00}h_{l}^{(k,1)}+\cdots +e_{n-1,0}h_{l}^{(k,n)}
\end{equation*}%
and 
\begin{equation*}
\Phi _{k}=e_{00}\Phi _{k}^{(1)}+\cdots +e_{n-1,0}\Phi _{k}^{(n)}.
\end{equation*}%
Therefore, $e_{j0}h_{l}^{(k,,j)}\left( B\right) \longrightarrow e_{j0}\Phi
_{k}^{(j)}\left( B\right) $ in $H^{2}$. But multiplication by $e_{j0}$ is an
isometry on $H^{2}\left( B\right) $, so we have $h_{l}^{(k,j)}\left(
B\right) \longrightarrow \Phi _{k}^{(j)}\left( B\right) $ in $H^{2}$. This
implies that $h_{l}^{(k,j)}\left( z\right) \longrightarrow \Phi
_{k}^{(j)}\left( z\right) $ in $H^{2}$. By Lemma 2, there exists a sequence $%
\left\{ g_{l}^{(k,j)}\right\} _{l=1}^{\infty }\subset H^{\infty }$ such that 
$g_{l}^{(k,j)}\left( z\right) h_{l}^{(k,j)}\left( z\right) \longrightarrow
\Phi _{k}^{(j)}\left( z\right) $ in $H^{p}$ as $l\rightarrow \infty $.
Define 
\begin{equation*}
g_{l}^{(k)}=g_{l}^{(k,1)}\left( B^{2}\right) \cdots g_{l}^{(k,n)}\left(
B^{2}\right) 
\end{equation*}%
so that $\left\{ g_{l}^{(k)}\right\} _{l=1}^{\infty }$ is uniformly bounded
and converges to $1~a.e.$ Consider 
\begin{eqnarray*}
g_{l}^{(k)}h_{l}^{(k)}
&=&\sum\limits_{j=1}^{n}e_{j-1,0}g_{l}^{(k)}h_{l}^{(k,j)} \\
&=&\sum\limits_{j=1}^{n}e_{j-1,0}g_{l}^{(k,1)}\left( B^{2}\right) \cdots
g_{l}^{(k,n)}\left( B^{2}\right) h_{l}^{(k,j)}
\end{eqnarray*}%
We now show that $g_{l}^{(k,1)}\left( B^{2}\right) \cdots
g_{l}^{(k,n)}\left( B^{2}\right) h_{l}^{(k,j)}\longrightarrow \Phi _{k}^{(j)}
${\large \ in }$H^{p}$. Note that the sequence $\theta
_{l}=g_{l}^{(k,1)}\left( B^{2}\right) \cdots g_{l}^{(k,j-1)}\left(
B^{2}\right) g_{l}^{(k,j+1)}\left( B^{2}\right) \cdots g_{l}^{(k,n)}\left(
B^{2}\right) $ is uniformly bounded and converges to $1a.e.$, and the
sequence $\psi _{l}=g_{l}^{(k,j)}h_{l}^{(k,j)}\longrightarrow \Phi _{k}^{(j)}
$ in $H^{p}$. It can be shown that $\theta _{l}\psi _{l}\longrightarrow \Phi
_{k}^{(j)}$ in $H^{p}$, and hence $g_{l}^{(k)}h_{l}^{(k)}\longrightarrow
\Phi _{k}$ in $H^{p}$. Now by the invariance of $M$ we have $\Phi _{k}\in M$%
. This means that $M$ can be written as 
\begin{equation*}
M=\varphi _{1}N_{1}\oplus \varphi _{2}N_{2}\oplus \cdots \oplus \varphi
_{n}N_{n},
\end{equation*}%
where 
\begin{equation*}
N_{j}=\left\{ h\in H^{p}\left( B\right) :\varphi _{j}h\in M\right\} 
\end{equation*}%
is a closed subspace of $H^{p}$. It is easy to see that $N_{j}$ is invariant
under $B^{2}$ and $B^{3}$, and is also dense in $H^{2}\left( B\right) $.%
\newline
Note that all $N_{j}s$ cannot be $B-$ invariant simultaneously. For if they
are then it would imply that $M$ is also $B-$invariant which is not
possible. Thus, some $N_{j}$ is not invariant under $B$. Without loss of
generality assume that $N_{1}$is not invariant. We show that even this is
not possible. For any $f\in H^{p}\left( B\right) $, we can find a sequence $%
\left\{ f_{l}\right\} $ in $N_{1}$ such that $f_{l}\rightarrow f$ in $H^{2}$
because $N_{1}$ is dense in $H^{2}\left( B\right) $. Once again we can write 
\begin{equation*}
f_{l}=e_{00}f_{l}^{(1)}+\cdots +e_{n-1,0}f_{l}^{(n)}
\end{equation*}%
and 
\begin{equation*}
f=e_{00}f^{(1)}+\cdots +e_{n-1,0}f^{(n)}
\end{equation*}%
so that $f_{l}^{(j)}\rightarrow f^{(j)}$ in $H^{2}$. Again by Lemma $3$,
there exists a sequence $\left\{ g_{l}^{(j)}\right\} _{l=1}^{\infty }$in $%
H^{\infty }$, such that $g_{l}^{(j)}f_{l}^{(j)}\rightarrow f^{(j)}$ in $H^{p}
$. Taking $g_{l}=g_{l}^{(1)}\left( B\right) \cdots g_{l}^{(n)}\left(
B\right) $, it follows that $B^{m}g_{l}f_{l}\rightarrow B^{m}f$ in $H^{p}$
for $m\geq 2$. Thus $B^{m}f\in N_{1}$ for $m\geq 2$. \newline
Let us choose $f=1$. So we have $B^{2}H^{p}\left( B\right) \subset N_{1}$.
This inclusion must be strict as $N_{1}$ is not invariant under $B$. So $%
N_{1}$ is of the form 
\begin{equation*}
N_{1}=A\oplus B^{2}H^{p}\left( B\right) ,
\end{equation*}%
where $A$ is a non zero subspace. We know that 
\begin{equation*}
A\subsetneq N_{1}\varsubsetneq H^{2}\left( B\right) .
\end{equation*}%
If $1$ and $B$ belong to $A$, then $N_{1}=H^{p}\left( B\right) $ which is
not possible.\newline
So $A=\left\langle \alpha +\beta B\right\rangle $, where $\alpha \neq 0$.
Again the density of $N_{1}$ implies that there exists a sequence $\left\{
\alpha _{n}\left( \alpha +\beta B\right) +B^{2}f_{n}\right\} \subset N_{1}$
that converges to $1$ in $H^{2}$. This gives 
\begin{eqnarray*}
B^{2}f_{n} &\rightarrow &0, \\
\alpha \alpha _{n} &\rightarrow &1,\text{ and} \\
\beta \alpha _{n} &\rightarrow &0
\end{eqnarray*}%
Therefore, $\beta =0$, which means that there cannot be a sequence in $N_{1}$
that converges to $B$ in $H^{2}$ norm. This contradicts the fact that $N_{1}$
is dense in $H^{2}\left( B\right) $. This contradiction stems from the fact
that $M_{1}$ is assumed to be invariant under $B$. Thus $M_{1}$ is invariant
under $B^{2}$ and $B^{3}$ but not under $B$. Now by the validity of our
result on $H^{2}$, there exist $B-$ inner functions $J_{1},\ldots ,J_{r}$, $%
r\leq n$, such that%
\begin{equation*}
M_{1}=\left( \sum\limits_{j=1}^{2r-1}\oplus \left\langle \varphi
_{j}\right\rangle \right) \oplus B^{2}J_{1}H^{2}\left( B\right) \oplus
\cdots \oplus B^{2}J_{r}H^{2}\left( B\right) 
\end{equation*}%
where $\varphi _{j}=\alpha _{1}^{j}J_{1}+\alpha _{2}^{j}J_{1}B+\alpha
_{3}^{j}J_{2}+\alpha _{4}^{j}J_{2}B+\cdots +\alpha _{2n-1}^{j}J_{n}+\alpha
_{2n}^{j}J_{n}B$. From the form of $\varphi _{j}$ it is clear that $\varphi
_{j}\in H^{p}$. Using the arguments already used in the proof it can be
shown that $\varphi _{j}\in M$. Also essentially repeating the arguments as
in the previous case we can easily establish that $B^{2}J_{1},\ldots
,B^{2}J_{r}\in M$. \newline
Thus 
\begin{equation*}
\left( \sum\limits_{j=1}^{2r-1}\oplus \left\langle \varphi _{j}\right\rangle
\right) \oplus B^{2}J_{1}H^{p}\left( B\right) \oplus \cdots \oplus
B^{2}J_{r}H^{p}\left( B\right) \subset M.
\end{equation*}%
To establish the reverse inclusion consider any $f\in M$. By virtue of the
characterization of $M_{1},$ $f=\alpha _{1}\varphi _{1}+\cdots +\alpha
_{2n-1}\varphi _{2n-1}+B^{2}J_{1}h_{1}+\cdots +B^{2}J_{n}h_{n}$. Note that $%
B^{2}J_{1}h_{1}+\cdots +B^{2}J_{r}h_{r}=f-\alpha _{1}\varphi _{1}-\cdots
-\alpha _{2r-1}\varphi _{2r-1}\in H^{p}$. Hence, by Lemma 4, $h_{1},\ldots
,h_{n}\in H^{p}\left( B\right) $. Hence $f\in \left(
\sum\limits_{j=1}^{2r-1}\oplus \left\langle \varphi _{j}\right\rangle
\right) \oplus B^{2}J_{1}H^{p}\left( B\right) \oplus \cdots \oplus
B^{2}J_{r}H^{p}\left( B\right) $ so that $M\subset \left(
\sum\limits_{j=1}^{2r-1}\oplus \left\langle \varphi _{j}\right\rangle
\right) \oplus B^{2}J_{1}H^{p}\left( B\right) \oplus \cdots \oplus
B^{2}J_{r}H^{p}\left( B\right) .$ This completes the proof of the theorem.
\end{proof}

\section{The $B-$ invariant subspaces of $H^{p}$}

\bigskip The ideas from the above proof can be applied to derive a new
factorization free proof of the following invariant subspace theorem
obtained in \cite{lance} for the cases $1\leq p\leq \infty $, $p\neq 2$. In
addition we have extended the theorem to the case $0<p<1.$ 
%%%%%%%%%%%%%%%%%%%%%%%%%%%%
%%%%%%%%%%%%%%%%%%%%%%%%%%%%
%%%%%%%%%%%%%%%%%%%%%%%%%%%%

\begin{theorem}
Let $M$ be a closed subspace of $H^{p}$, $0<p\leq \infty $, $p\neq 2$, such
that $M$ is invariant under $H^{\infty }\left( B\right) $. Then there exist $%
B-$ inner functions $J_{1},\ldots ,J_{r}$, $r\leq n$, such that 
\begin{equation*}
M=J_{1}H^{p}\left( B\right) \oplus \cdots \oplus J_{r}H^{p}\left( B\right) .
\end{equation*}%
When $0<p<1$, then, as the proof will show, the right hand side is to be
read as being dense in $M$ i.e. its closure in the $H^{p}$ metric is all of $%
M$.
\end{theorem}

\begin{proof}
The idea of the proof is quite similar to the proof of the Theorem 3. We
shall only sketch the details. Using the fact that every $0\neq f\in H^{p}$
can be written as a product of an appropriate number of $H^{2}$ functions,
we can construct an outer function $O(z)$, in a manner identical to the
proof of Theorem 3, such that $O(B(z))f\in M\cap H^{2}$. Thereby
establishing that $M\cap H^{2}\neq \{0\}$. Now $M\cap H^{2}$ is a closed
subspace of $H^{2}$ and invariant under $H^{\infty }(B)$, so by Theorem 2
there exist $B-$ inner functions $J_{1},\ldots ,J_{r}$, with $r\leq n$, such
that 
\begin{equation*}
M\cap H^{2}=J_{1}H^{2}\left( B\right) \oplus \cdots \oplus J_{r}H^{2}\left(
B\right) .
\end{equation*}

\noindent Next we show that $M\cap H^{2}$ is dense in $M$. It is trivial to
note that $\overline{M\cap H^{2}}\subseteq M$, where the bar denotes closure
in $H^{p}$. In order to establish the reverse inequality, we follow the same
arguments used in the proof of Theorem 3 to construct a sequence of outer
functions $\{O_{l}(z)\}_{l=0}^{\infty }$, such that for any $f\in M$, $%
O_{l}(B(z))f\in M\cap H^{2}$, and $O_{l}(B(z))f\rightarrow f$ in $H^{p}$, as 
$l\rightarrow \infty $. The proof of Theorem 3 also establishes that, for $%
0<p<1$, $\overline{M\cap H^{2}}$ takes the form $\overline{J_{1}H^{p}\left(
B\right) \oplus \cdots \oplus J_{r}H^{p}\left( B\right) \text{ }}$ and for $%
1\leq p<2,$ $\overline{M\cap H^{2}}=$ $J_{1}H^{p}\left( B\right) \oplus
\cdots \oplus J_{r}H^{p}\left( B\right) $ thereby establishing the
characterization for $M$ in these cases. Next we deal with the case when $%
2<p\leq \infty .$ As in the proof of Theorem 3, we consider $M_{1}=\overline{%
M}$, the closure of $M$ in $H^{2}$. Since $M$ is $B-$ invariant, we have $%
M_{1}$ is $B-$ invariant. So by Theorem 2 there exist $B-$ inner functions $%
J_{1},\ldots ,J_{r}$, $r\leq n$, such that 
\begin{equation*}
M_{1}=J_{1}H^{2}(B)\oplus \cdots \oplus J_{r}H^{2}(B).
\end{equation*}%
Thus any arbitrary $f\in M$ can be written as $f=J_{1}h_{1}+\cdots
+J_{r}h_{r}$, for some $h_{1},\ldots ,h_{r}\in H^{2}$. By Lemma $4$, these $%
h_{1},\ldots ,h_{r}\in H^{p}$, and hence $M\subset J_{1}H^{2}(B)\oplus
\cdots \oplus J_{r}H^{2}(B)$. To prove the inclusion in the reverse, we need
to establish that, for each $k=1,\ldots ,r$, $J_{k}H^{p}(B)\subset M$. For
an arbitrary $h\in H^{p}(B)$, consider $\Psi _{k}=J_{k}h$. Now proceeding in
the same fashion as in the proof of Theorem 3 (where we show that $\Phi
_{k}\in M$), it follows that $\Psi _{k}\in M$.
\end{proof}

\begin{acknowledgement}
The second author thanks Vern Paulsen for useful discussions. The first
author thanks the Shiv Nadar University, Dadri, Uttar Pradesh, and the
Mathematical Sciences Foundation, New Delhi, for the facilities given to
complete this work.
\end{acknowledgement}


\begin{thebibliography}{99}
\bibitem{agler} J. Agler and J.E. McCarthy, \textit{Pick Interpolation and
Hilbert Function Spaces,} Graduate Studies in Mathematics, 44, American
Mathematical Society, Providence, RI, 2002.

\bibitem{barb} C. Cowen and B. Mccluer, \textit{Composition Operators on
Spaces of Analytic Functions,} CRC Press, 1994.

\bibitem{dh} K. R. Davidson, and R. Hamilton, Nevanlinna-Pick interpolation
and factorization of linear functionals, Integral Equations and Operator
theory, 70(2011),125-149.

\bibitem{dprs} K.R. Davidson, V.I. Paulsen, M. Raghupathi, and D. Singh, A
constrained Nevanlinna-Pick interpolation problem, Indiana University
Mathematics Journal, 58(2009), 709--732.

\bibitem{dur} P.L. Duren, \textit{Theory of }$H^{p}$\textit{\ Spaces},
Academic Press, London-New York, 1970.

\bibitem{gam} T.W. Gamelin, \textit{Uniform Algebras,} AMS Chelsea 1984.

\bibitem{gar} J.B. Garnett, \textit{Bounded Analytic Functions}, Academic
Press, 1981.

\bibitem{hal} P.R. Halmos, Shifts on Hilbert spaces, J. Reine Angew. Math.
208(1961), 102-112.

\bibitem{hel} H. Helson, \emph{Harmonic analysis}, Hindustan Book Agency,
1995.

\bibitem{hoff} K. Hoffman, \textit{Banach Spaces of Analytic Functions},
Prentice Hall, 1962.

\bibitem{jury} M. Jury, G. Knese, and S. McCullough, Nevanlinna-Pick
interpolation on distinguished varieties in the bidisk, J. Funct. Anal.,
262(2012), 3812-3838.

\bibitem{koosis} P. Koosis, \textit{Introduction to }$H^{p}$\textit{\ spaces,%
} Cambridge University Press, Cambridge, 1998.

\bibitem{lance} T.L. Lance and M.I. Stessin, Multiplication Invariant
Subspaces of Hardy Spaces, Can. J. Math, 49(1997) 100-118.

\bibitem{lax} P.D. Lax, Translation invariant spaces, Acta Math. 101(1959),
163-178.

\bibitem{mc} S. McCullough, and T.T. Trent, Invariant subspaces and
Nevanlinna-Pick Kernels, J. Funct. Anal., 178(2000), 226-249.

\bibitem{mrin} M. Raghupathi, Abrahamse's interpolation theorem and Fuchsian
groups, J. Math. Anal. Appl,, J. Math. Anal. Appl., 355 (2009), 258 --- 276.

\bibitem{mrin2} M. Raghupathi, Nevanlinna-Pick interpolation for $%
C+BH^{\infty }$, Integral Equations and Operator theory, 63 (2009), 103-125.

\bibitem{paul} V.I. Paulsen and D. Singh, Modules over subalgebras of the
disk algebra, Indiana Univ. Math. Jour. 55 (2006), 1751-1766.

\bibitem{sinth} D. Singh and V. Thukral, Multiplication by finite Blaschke
factors on de Branges spaces, J. Operator Theory 37(1997), 223-245.

\bibitem{nik} N. K. Nikolski, \emph{Operators, Functions and Systems: an
easy reading}, Vol. 1, Amer. Math. Soc., 2002.

\bibitem{RS} M. Raghupathi, and D. Singh, Function theory in real Hardy
spaces, Math. Nachr., 284 (2011), 920-930.
\end{thebibliography}
\end{document}